\documentclass[11pt,oneside,reqno]{amsart}

\usepackage[utf8]{inputenc}
\usepackage[T1]{fontenc}
\usepackage[ngerman,english]{babel}
\usepackage{lmodern}
\usepackage{mathtools,amsthm,amsfonts,enumerate,tensor,slashed,mathrsfs,mathabx,bbm}
\usepackage{enumitem}
\usepackage{tikz}
\usepackage{bookmark}
\usepackage{hyperref}
\usepackage{cite}

\usepackage{graphicx}

\newtheorem{Thm}{Theorem}[section]
\newtheorem{Le}[Thm]{Lemma}
\newtheorem{Prop}[Thm]{Proposition}
\newtheorem{Cor}[Thm]{Corollary}

\theoremstyle{definition}
\newtheorem{Def}[Thm]{Definition}

\theoremstyle{remark}

\numberwithin{equation}{section}

\renewcommand{\d}{\mathrm{d}}
\renewcommand{\b}{\beta}
\renewcommand{\a}{\alpha}
\def\and{\quad \text{and} \quad}
\newcommand{\g}{\mathfrak{g}}
\newcommand{\h}{\mathfrak{h}}
\renewcommand{\P}{\mathcal{P}}
\newcommand{\Q}{\mathcal{Q}}
\newcommand{\R}{\mathcal{R}}
\newcommand{\C}{\mathcal{C}}
\newcommand{\B}{\mathfrak{B}}
\newcommand{\m}{\mathrm{m}}
\newcommand{\so}{\mathfrak{so}}
\newcommand{\gl}{\mathfrak{gl}}

\allowdisplaybreaks

\begin{document}

\begin{flushright}
ITP--UH--18/14
\end{flushright}

\vspace{0.5cm}
\title{A method of deforming G-structures}
\keywords{G-structures; instantons in higher dimensions; connections on principal bundles; flux compactifications}

\author{Severin Bunk}
\address{Institut für Theoretische Physik, Leibniz Universität Hannover}
\curraddr{Department of Mathematics, Heriot-Watt University
Colin Maclaurin Building, Riccarton, Edinburgh EH14 4AS, U.K.
}

\email{sb11@hw.ac.uk}

\begin{abstract}
We consider deformations of $G$-structures via the right action on the frame bundle in a base-point-dependent manner.
We investigate which of these deformations again lead to $G$-structures and in which cases the original and the deformed $G$-structures define the same instantons.
Further, we construct a bijection from connections compatible with the original $G$-structure to those compatible with the deformed $G$-structure and investigate the change of intrinsic torsion under the aforementioned deformations.
Finally, we consider several examples.
\end{abstract}

\maketitle

\vspace{-1cm}

\section{Introduction}

In recent years, $G$-structures have become a subject of intense research for several reasons.
They provide the foundation of the classification of special geometries, as for instance Calabi-Yau, nearly Kähler and half-flat spaces.
Interest in Riemannian manifolds of special holonomy has been increasing since the publication of Berger's famous list of Riemannian holonomy groups~\cite{Berger55}.
Such manifolds can be seen as manifolds endowed with a $G$-structure which is compatible with the Levi-Civita connection, i.e. whose intrinsic torsion with respect to the Levi-Civita connection vanishes.
Sometimes such $G$-structures are called integrable.

With the success of the utilisation of instantons in the classification of $4$-dimensional manifolds~\cite{Donaldson83}, the relevance of the investigation of instantons in more general situations, as initiated in~\cite{DonaldsonThomas98,DonaldsonSegal11}, has been realised.
The definition of instantons in higher dimensions crucially relies on the notion of $G$-structures.
Moreover, $G$-structures and instantons are still becoming increasingly prominent among physicists due to their natural occurrence in string theory (cf.~\cite{Strominger:1986,HarlandNoelle,Grana,Ivanov,Gemmer:2012pp,Haupt:2014ufa,Cardosoea,Gauntlett:2002,Gauntlett:2003} and references therein).

Regarding any of these aspects, it is desirable to consider explicit examples of manifolds with $G$-structures.
In this note, we present a general scheme for the construction of certain families of $G$-structures from a given $G$-structure.
We investigate the space of connections compatible with these structures as well as the change of intrinsic torsion, briefly consider the induced instanton conditions and finish by introducing some examples.

\section{Preliminaries}
First, we focus on the geometry of principal bundles and principal subbundles.
Notions and notation used here can be found for example in~\cite{Baum,KobNomI63}.

Let $(\P,\pi,M,H)$ be a principal bundle with total space $\P$, structure group $H$ and projection $\pi$ over a $D$-dimensional base manifold $M$.
We denote the right-action of $H$ on $\P$ by
\begin{equation}
 R: \P {\times} H \rightarrow \P,\ (p,h) \mapsto R_h\, p.
\end{equation}
Representations $\rho: H \rightarrow GL(V)$ of $H$ on a (finite-dimensional) vector space $V$ give rise to associated vector bundles
\begin{equation}
 E = \P \times_{(H,\rho)} V = (\P {\times} V)/{\sim}\,,
\end{equation}
which consist of the equivalence classes
\begin{equation}
 \big[ p,v \big] = \big[ R_h\, p,\, \rho(h^{-1})(v) \big] \quad \forall\, h \in H.
\end{equation}
We call a $k$-form on $\P$ \emph{horizontal} if its evaluation on tangent vectors vanishes as soon as any of the vectors is vertical.
A $V$-valued $k$-form $\omega$ on $\P$ is said to be \emph{of type $\rho$} if ${R_h}^*\, \omega = \rho(h^{-1}) \circ \omega$.
We denote the space of horizontal $k$-forms of type $\rho$ with values in $V$ by $\Omega^k_{hor}(\P,V)^{(H,\rho)}$.

Consider a principal $G$-subbundle $(\Q,\pi,M,G)$ of $\P$.
That is, $\Q \subset \P$ is a submanifold of $\P$, and the $G$-action on $\Q$ coincides with the restriction of the right-action of $H$ on $\P$ to $G \subset H$ and $\Q \subset \P$.
In this case, every $[p,v] \in E$ can be written in the form $[q,v']$ for some $q \in \Q$ due to the transitivity of the $H$-action on the fibres of $\P$.
Thus,
\begin{equation}
 E = \P \times_{(H,\rho)} V = \Q \times_{(G,\rho)} V,
\end{equation}
where the representation of $G$ on $V$ is the restriction of $\rho$ to $G$.\\
If $G \subset H$ is given as the stabiliser of an element $\tau_0 \in V$, one can prove the following lemma utilised e.g. in~\cite{ContiSalamon}:

\begin{Le}
\label{st:G-structures_and_defining_sections}
 Let $G \subset H$ be the stabiliser of $\tau_0 \in V$ in the representation $\rho$ of $H$ on $V$.
 There exists a one-to-one correspondence between principal $G$-subbundles $\Q \subset \P$ and global sections $\tau \in \Gamma(E)$ with the following property:
 One can find an open covering $\{(U_i,s_i)\}_{i \in \Lambda}$ of $M$ by local sections of $\P$ such that for every $i \in \Lambda$
 \begin{equation}
 \label{eq:def_sec:local_rep}
  \tau_x = [s_i(x),\, \tau_0] \quad \forall\, x \in U_i.
 \end{equation}
\end{Le}

Principal $G$-subbundles of $\P$ are more commonly characterised by sections of the associated fibre bundle $\P \times_{(H,\ell)} (H/G)$, where $\ell$ is the natural left-action of $H$ on $H/G$.
However, the characterisation introduced in Lemma~\ref{st:G-structures_and_defining_sections} is more convenient in the context of this paper.
We will refer to a section $\tau$ as in Lemma~\ref{st:G-structures_and_defining_sections} as a \emph{defining section} for $\Q$.

We denote the space of connections on $\P$ by $\C(\P)$ and frequently identify a connection with its connection $1$-form.
From~\cite{Baum,KobNomI63} we take the following assertion:

\begin{Le}
\label{st:fhat_fcheck}
 For $\P$ and $\Q$ as above, the following statements hold true:
 \begin{enumerate}
  \item Every connection $A \in \C(\Q)$ extends to a unique connection on $\P$.
   This yields an embedding
   \begin{equation}
    \widehat{f}_\Q: \C(\Q) \hookrightarrow \C(\P).
   \end{equation}
  \item Assume, that there exists a $G$-invariant splitting $\h = \g \oplus \m$, where $\h$ and $\g$ are the Lie algebras of $H$ and $G$, respectively, and where $G$ acts via the restriction of the adjoint representation of $H$.
  Then, there exists a map
   \begin{equation}
    \widecheck{f}_\Q: \C(\P) \rightarrow \C(\Q),\quad A \mapsto pr_\g \circ A_{\vert \Q}.
   \end{equation}
 \end{enumerate}
\end{Le}
Note that the first map is injective, while the second is surjective.

\section{Normal deformations}
\label{sect:ndefs}

\subsection{Principal subbundles}
Manifolds often admit more than one additional geometric structure, and often some of these structures may be obtained as deformations of others.
For example, a Riemannian manifold $(M,g)$ does not only admit the Riemannian metric $g$, but also all the Riemannian metrics $\phi\, g$, where $\phi \in C^\infty(M,\mathbb{R}_+)$ is a positive function on $M$.

Let $\Q$ be a principal $G$-subbundle of a principal $H$-bundle $\P$ as in the previous section.
As this is the general framework behind $G$-structures, we are tempted to ask for general ways to deform $\Q$ such that we again arrive at principal $G$-subbundles $\Q'$ of $\P$.

One possibility to obtain candidates for such subbundles is to deform the submanifold $\Q \subset \P$ using the right-action of $H$ on $\P$.
To this end, consider a map $h \in C^\infty(M,H)$, i.e.~$h: M \rightarrow H,\, x \mapsto h(x)$.
Every such map induces a diffeomorphism
\begin{equation}
 \R_h:\, \P \rightarrow \P, \quad p \mapsto R_{h(\pi(p))}\, p.
\end{equation}
Note that in general this is not an automorphism of $\P$ as a principal bundle, since $\R$ does not commute with the right $H$-action.
Rather we have
\begin{equation}
 \R_h \circ R_a = R_{\alpha(h^{-1})(a)} \circ \R_h,
\end{equation}
where $\a(a_1)(a_2) \coloneq a_1\, a_2\, a_1^{-1}$ denotes the inner automorphism of $H$.
Defining $(h_1\, h_2)(x) = h_1(x)\, h_2(x)$ endows $C^\infty(M,H)$ with a group structure, and $h \mapsto \mathcal{R}_h$ provides a right-action of $C^\infty(M,H)$ on $\P$.

As $\mathcal{R}_h$ is a diffeomorphism, the image of a submanifold of $\P$ under this map is, again, a submanifold of $\P$.
Therefore, $\Q' \coloneq \mathcal{R}_h\, \Q$ is a natural candidate for a new principal $G$-subbundle of $\P$.
As $\Q'$, hence, is a submanifold of $\P$ by construction, we have to ensure that $\Q'$ is a principal $G$-bundle and, additionally, a principal subbundle of $\P$.
The result is the following statement:

\begin{Thm}
\label{st:rotchar}
 Let $\Q \subset \P$ be a principal $G$-subbundle of $(\P,\pi,M,H)$, and consider a map $h \in C^\infty(M,H)$.
 Then the following statements are equivalent:
 \begin{enumerate}
  \item $\Q' = \mathcal{R}_h\, \Q$ is a  principal $G$-subbundle of $\P$.
  \item $h$ takes values in $N_H(G)$ only, where
   \begin{equation}
    N_H(G) \coloneq \big\{ a \in H\, \big\vert\, a\, g\, a^{-1} \in G\ \, \forall\, g \in G \big\}
   \end{equation}
   is the normaliser of $G$ in $H$.
  \item If $\Q$ has a defining section $\tau \in \Gamma(E)$ in a vector bundle $E$ associated to $\P$ as above, the prescription
    \begin{equation}
     \tau'_{\pi(q)} \coloneq \big[ R_{h(\pi(q))}\, q,\, \tau_0 \big]
    \end{equation}
   yields a defining section for $\Q'$.
 \end{enumerate}
\end{Thm}

\begin{proof}
 $(1) \Rightarrow (2)$:
  Assume that $\Q' = \mathcal{R}_h\, \Q \subset \P$ is a  principal $G$-subbundle of $\P$.
  In particular, $\Q'$ carries a right-action of $G$ given by the restriction of the right-action of $H$ on $\P$, and $\Q'$ is invariant under this right-action of $G$.
  That is,
  \begin{equation}
   R_g\, q' \in \Q' \quad \forall\, g \in G,\, q' \in \Q'.
  \end{equation}
  Now, for every $q' \in \Q'$ there is a unique $q \in \Q$ such that $q' = R_{h(x)}\, q$, where $x = \pi(q')$.
  Hence, we have
  \begin{equation}
   R_g\, q' = R_g\, R_{h(x)}\, q = R_{h(x)}\, R_{\a(h(x))(g)}\, q \in \Q'
  \end{equation}
  for all $q \in \Q$ and $g \in G$.
  
  However, since $\Q'$ is defined as the image of $\Q$ under the diffeomorphism $\mathcal{R}_h$, the statement that $R_g\, q' \in \Q'$ is true if and only if
  \begin{equation}
   R_{\a(h(x))(g)}\, q \in \Q.
  \end{equation}
  This is equivalent to
  \begin{equation}
   \a(h(x))(g) \in G \quad \forall\, x \in M,
  \end{equation}
  because $\Q$ is invariant under the right-action of $G$, and because the right-action of $G$ on $\Q$ is simply transitive on the fibres.
  In turn, this is equivalent to $h(x) \in N_H(G)$ for all $x \in M$.\medskip
  
 $(2) \Rightarrow (1)$:
  First, note that, as $\mathcal{R}_h$ is a diffeomorphism on $\P$, the set $\Q'$ endowed with the induced differentiable structure is, indeed, a submanifold of $\P$.
  Moreover, it is a fibre bundle over $M$ with typical fibre $G$.
  For this to be a principal $G$-subbundle of $\P$, the restriction of the right-action of $H$ on $\P$ has to be a right-action on $\Q'$ as well.
  As before, this right action is given by
  \begin{equation}
  \label{eq:right_action_Q'}
   R_g\, q' = R_{h(x)}\, R_{\a(h(x))(g)}\, q.
  \end{equation}
  This is an element of $\Q'$ since we assume statement $(2)$ of the proposition, and, moreover, \eqref{eq:right_action_Q'} defines a right-action of $G$ on $\Q'$.
  Furthermore, as the right-action of $G$ on $\Q$ is simply transitive on the fibres and $\a(h(x))$ is an automorphism of $G$ for every $x \in M$, the above $G$-action on $\Q'$ is simply transitive on the fibres of $\Q'$.
  Hence, $\Q'$ is a principal $G$-subbundle of $\P$.\medskip
  
 $(2) \Leftrightarrow (3)$:
  In the case of $(2)$ we have
  \begin{equation}
   \tau'_{\pi(q)} \coloneq [R_{h(x)}\, q,\, \tau_0] = [q,\, \rho(h(x))(\tau_0)].
  \end{equation}
  This is a well-defined element of the fibre of $E$ at $\pi(q)$ if and only if it is independent of the particular choice of $q$.
  Thus, we compute
  \begin{align}
   [R_{h(x)}\, R_g\, q,\, \tau_0] ={}& [R_{\a(h(x)^{-1})(g)}\, R_{h(x)}\, q,\, \tau_0]\\
    ={}& [R_{h(x)}\, q,\, \rho(\a(h(x)^{-1})(g))(\tau_0)], \notag
  \end{align}
  which is equal to $[R_{h(x)}\, q,\, \tau_0]$ if and only if
  \begin{equation}
   \rho \big( \a(h(x)^{-1})\,(g) \big) (\tau_0) = \tau_0.
  \end{equation}
  Again, this is the requirement that $\a(h(x)^{-1})(g) \in G$, or, equivalently, that $h(x) \in N_H(G)$ for all $x \in M$.
  As $h$ is smooth and globally well-defined, so is $\tau'$.
\end{proof}

\noindent
In particular, Theorem~\ref{st:rotchar} implies

\begin{Cor}
 The map
 \begin{equation}
  \mathcal{R}: G\P \times C^\infty \big(M,N_H(G) \big) \rightarrow G\P, \quad (\Q,h) \mapsto \mathcal{R}_h\, \Q
 \end{equation}
 defines a right-action of $C^\infty(M,N_H(G))$ on the space $G\P$ of principal $G$-subbundles of $\P$.
\end{Cor}

\begin{Def}
 We say that two principal $G$-subbundles $(\Q,\Q')$ satisfy the \emph{normal deformation property} if there exists an $h \in C^\infty(M,N_H(G))$ such that $\Q' = \mathcal{R}_h \Q$, i.e.~if they are related by the right-action of $C^\infty(M,N_H(G))$ on $G\P$.
\end{Def}

\noindent
Note that $\Q'$ can always be endowed with a right-action of $G$ by defining
\begin{equation}
 R'_g\, q' \coloneq R_{\alpha(h^{-1})(g)}\, q'.
\end{equation}
However, in general this does not coincide with the restriction of the $H$-action on $\P$ as it is necessary for $\Q'$ to be a principal subbundle of $\P$.
These two right-actions of $G$ on $Q'$ agree if and only if $h \in C^\infty(M,C_H(G))$, i.e. if and only if $h$ takes values in the centraliser of $G$ exclusively.

\subsection{Connections}
An important set of data associated to a principal bundle is the set of its connections.
If there is a $G$-invariant splitting $\h = \g \oplus \m$, the aforementioned right-action of $C^\infty(M,N_H(G))$ on $G\P$ transforms this data in a well-controlled manner.

\begin{Thm}
\label{st:general_isomp_of_connections}
 Let $\Q$ be a principal $G$-subbundle of $\P$, and let there be a $G$-invariant splitting $\h = \g \oplus \m$.
 Then, for any $h \in C^\infty(M,N_H(G))$ there exists a bijective map
 \begin{equation}
  f_{\Q,h}: \C(\Q) \rightarrow \C(\Q'), \quad A \mapsto \widecheck{f}_{\Q'} \circ \widehat{f}_\Q\, (A),
 \end{equation}
 where $\Q' = \mathcal{R}_h \Q$ as before.
\end{Thm}

\begin{proof}
 Lemma~\ref{st:fhat_fcheck} directly implies that $f_{\Q,h}(A) \in \C(\Q')$ is a connection on $\Q'$.
 In order to prove that $f_{\Q,h}$ is bijective, recall that $\C(\Q)$ is an affine vector space modelled over the vector space $\Omega^1_{hor}(\Q,\g)^{(G,Ad)}$ (cf.~\cite{Baum}, for instance).
 Due to its $Ad$-equivariance, every $\omega \in \Omega^1_{hor}(\Q,\g)^{(G,Ad)}$ can be extended to a horizontal, $Ad$-equivariant $1$-form $\widehat{\omega} \in \Omega^1_{hor}(\P,\h)^{(H,Ad)}$.
 Its restriction to $\Q'$ is again horizontal and $Ad$-equivariant.
 Moreover, it is $\g$-valued on $\Q'$.
 This can be seen either by using local sections or directly from the construction of $\widehat{\omega}$.
 
 In particular, on $\Omega^1_{hor}(\Q,\g)^{(G,Ad)}$ the procedure of extending to $\P$ and restricting to $\Q'$ is linear and has an inverse given by applying the same procedure starting from $\Q'$.
 Therefore, the auxiliary map given by
 \begin{equation}
  f: \Omega^1_{hor}(\Q,\g)^{(G,Ad)} \rightarrow \Omega^1_{hor}(\Q',\g)^{(G,Ad)},\quad \omega \mapsto \widehat{\omega}_{\vert \Q'}
 \end{equation}
 is an isomorphism of vector spaces.
 
 Now, consider a connection $A \in \C(\Q)$.
 Every connection on $\Q$ is of the form $A + \omega$ for an $\omega \in \Omega^1_{hor}(\Q,\g)^{(G,Ad)}$.
 Upon application of $f_{\Q,h}$ we obtain
 \begin{equation}
  f_{\Q,h}(A + \omega) = f_{\Q,h}(A) + f(\omega) = f_{\Q,h}(A) + \widehat{\omega}_{\vert \Q'}.
 \end{equation}
 Since $f_{\Q,h}(A) \in \C(\Q')$ and $f$ is bijective, $f_{\Q,h}$ is bijective as well.
 Note that from this equation we also see that $f_{\Q,h}$ is still defined independently of the choice of $A$.
\end{proof}

\noindent
In local representations, this means the following:
For every local section $s \in \Gamma(U,\Q)$ we obtain a local section $s' = R_h \circ s \in \Gamma(U,\Q')$, and for any connection $A \in \C(\Q)$ there exists a connection $f_{\Q,h}(A) \in \C(\Q')$ with local representation
\begin{equation}
\label{eq:f_Q,h(A)_local_rep}
 \begin{aligned}
  s'^* \big( f_{\Q,h}(A) \big) ={}& pr_{\g} \circ s'^*\, \widehat{f}_\Q(A)\\
   ={}& Ad(h^{-1}) \circ s^*A + pr_\g \circ h^*\mu_H,
  \end{aligned}
\end{equation}
where $\mu_H \in \Omega^1(H,\h)$ is the Maurer-Cartan form of the Lie group $H$.
We can drop the projection in the first term since $s^* A$ is $\g$-valued and $Ad(h^{-1})$ preserves $\g$.
Note that, due to the projection in the second term, the extensions of $A$ and $f_{\Q,h}(A)$ to connections on $\P$ differ in general.
Explicitly, we have
\begin{equation}
 s'^* \widehat{f}_\Q(A) = Ad(h^{-1}) \circ s^*A + h^*\mu_H,
\end{equation}
and, therefore,
\begin{equation}
\label{eq:A'-A}
 s'^* \big(\widehat{f}_\Q(A) - f_{\Q,h}(A) \big) = pr_\m \circ h^*\mu_H
\end{equation}
is the obstruction preventing the extension of $A$ to $P$ from restricting to a connection on $\Q'$
\footnote{The independence of the right-hand side of \eqref{eq:A'-A} of the choice of a local section $s' \in \Gamma(U,\Q')$ may be checked directly, using the fact that $X \in Lie(N_H(G))$ if and only if it satisfies $Ad(g^{-1})(X) - X \in \g \ \forall\, g \in G$.}.
Therefore, we have

\begin{Prop}
\label{st:general_isomp_algebraic}
 The map $f_{\Q,h}: \C(\Q) \rightarrow \C(\Q')$ introduced in Theorem~\ref{st:general_isomp_of_connections} is given explicitly by
 \begin{equation}
 \widehat{f}_\Q(A) - f_{\Q,h}(A) = \zeta_h \quad \forall\, A \in \C(\Q),
\end{equation}
where $\zeta_h \in \Omega^1_{hor}(\P,\h)^{(H,Ad)}$ is the unique horizontal $1$-form of type $Ad$ on $\P$ having local representations
\begin{equation}
 s'^* \zeta_h = pr_\m \circ h^*\mu_H \quad \forall\, s' \in \Gamma(U,\Q').
\end{equation}
\end{Prop}

Furthermore, note that the local representations $s'^* ( f_{\Q,h}(A) )$ transform in the desired way upon changing the local section $s'(x) \mapsto R_{g(x)} \circ s'(x)$.\medskip

For the more restrictive case of $h \in C^\infty(M,C_H(G)) \subset C^\infty(M,N_H(G))$, i.e.~$h$ taking values in the centraliser of $G$ in $H$, there is a simpler bijection of connections.
In this case, $\R_h: \Q \rightarrow \Q'$ is an isomorphism of principal fibre bundles, and we obtain a subclass of the deformations of $G$-structures investigated in~\cite{Stock}.
We can use the pullback to transport connections between $\Q$ and $\Q'$.
(Note that we might still have $\Q \neq \Q'$ if $C_H(G) \nsubset G$.)
This directly yields

\begin{Le}
\label{st:central_isomorphism_of_connections}
 If $h \in C^\infty(M,C_H(G))$, and $\Q$, $\Q'$, $s$ and $s'$ are as above, then for every $A \in \C(\Q)$ there exists an $A' = {\R_{h^{-1}}}^* A \in \C(\Q')$.
 In particular, all local representations of the connection forms and their field strengths with respect to $s$ and $s'$ coincide.
 That is,
 \begin{equation}
  s'^*A' = s^*A \and s^*F^A = s'^* F^{A'} \quad \forall\, s \in \Gamma(U,\Q).
 \end{equation}
\end{Le}

\noindent
This is, of course, very easy to use in explicit computations.

\subsection{Intrinsic torsion}
Most important geometric features of metric $G$-structures (i.e. where $G \subset SO(D)$ and $P = F(M)$, see Section~\ref{sect:G_structures_and_instantons}) are governed by their intrinsic torsion~\cite{GrayHervella,Chiossi:2002tw}.
Therefore, it seems worthwhile to investigate the change of intrinsic torsion of a $G$-structure upon application of normal deformations.
We again assume that we can split $\h = \g \oplus \m$ in a $G$-invariant manner.
Generalising the notion of intrinsic torsion of a metric $G$-structure (cf. e.g.~\cite{Alexandrov2005}), to any principal $G$-subbundle of $\P$ and a connection $A_0 \in \C(\P)$ we can assign the \emph{intrinsic torsion of $\Q$ with respect to $A_0$}, given by
\begin{equation}
\label{eq:inttor}
 T_\Q(A_0)_{\vert \Q} = pr_\m \circ A_{0\,\vert \Q} \ \in \Omega^1_{hor}(\Q,\m)^{(G,Ad)}.
\end{equation}
Since this is a $1$-form of type $Ad$ on $\Q$, it extends uniquely to a $1$-form $\widehat{T}_\Q(A_0) \in \Omega^1_{hor}(\P,\h)^{(H,Ad)}$ just as in the proof of Theorem~\ref{st:general_isomp_of_connections}.
For metric $G$-structures one takes $A_0$ to be the Levi-Civita connection of the ambient orthonormal frame bundle $SO(M,g)$.
However, note that a normal deformation of a metric $G$-structure may change the metric that the $G$-structure is compatible with if $h$ is not $SO(D)$-valued.
Therefore, the intrinsic torsions of these $G$-structures will, in general, be defined with respect to Levi-Civita connections of different metrics.

However, in order to really be able to compare these forms in terms of their interpretation as intrinsic torsion, it seems more appropriate to consider their corresponding bundle-valued $1$-forms $\tau_\Q(A_0) \in \Omega^1(M,Ad(\Q))$
\footnote{One can indeed show that the adjoint bundles of $\Q$ and $\Q'$ coincide as subbundles of $Ad(\P)$ if $(\Q,\Q')$ satisfy the normal deformation property.}.
Recall that for any local section $s \in \Gamma(U,\Q)$,
\begin{equation}
 \tau_\Q(A_0)_{\vert U} = \big[ s, s^* \big( T_\Q(A_0) \big) \big].
\end{equation}
In this language we may deduce how normal deformations change the intrinsic torsion.

\begin{Prop}
\label{st:inttor}
 Let $(\Q,\Q')$ be principal $G$-subbundles of $\P$ satisfying the normal deformation property with respect to $h \in C^\infty(M,N_H(G))$, and where there is a $G$-invariant splitting $\h = \g \oplus \m$.
 For $A_0, A_0' \in \C(\P)$ we have that
 \begin{align}
 \label{eq:inttor_and_nedef}
  \big( \tau_{\Q'}(A_0') - \tau_\Q(A_0) \big)_{\vert U}
   = \big[ s,\, &{}Ad(h) \circ pr_\m \circ Ad(h^{-1}) \circ s^*A_0' - pr_\m \circ s^*A_0 \notag\\
    &{}+ Ad(h) \circ pr_\m \circ h^*\mu_H \big]
 \end{align}
 for all $s \in \Gamma(U,\Q)$.
 
 In particular, if we consider $A_0' = A_0$ and assume that $\tau_{\Q}(A_0) = 0$, then
 \begin{equation}
  \big( \tau_{\Q'}(A_0') - \tau_\Q(A_0) \big)_{\vert U} = [ s', s'^* \zeta_h ] \quad s \in \Gamma(U,\Q),
 \end{equation}
 where $\zeta_h \in \Omega^1_{hor}(\P,\h)^{(H,Ad)}$ is as in Proposition~\ref{st:general_isomp_algebraic}.
\end{Prop}

\begin{proof}
 Consider local sections $s$ of $\Q$ and $s' = R_h \circ s$ of $\Q'$ as before.
 We compute
 \begin{align}
  \tau_{\Q'}(A_0')_{\vert U} ={}& \big[ s,\, s^* \big( \widehat{T}_{\Q'}(A_0') \big) \big] \notag\\
   ={}& \big[ s,\, s'^* \big( Ad(h) \circ T_{\Q'}(A_0') \big) \big] \\
   ={}& \big[ s,\,  Ad(h) \circ pr_\m \circ s'^*A_0' \big] \notag\\
   ={}& \big[ s,\,  Ad(h) \circ pr_\m \circ Ad(h^{-1}) \circ s^*A_0' + Ad(h) \circ pr_\m \circ h^*\mu_H \big] \notag
 \end{align}
 This implies the first statement.
 
 If $A_0' = A_0$, and $A_0$ restricts to a connection on $\Q$, $s^*A_0' = s^*A_0$ is $\g$-valued.
 Therefore, $Ad(h^{-1}) \circ s^*A_0'$ is $\g$-valued whence it is annihilated by $pr_\m$, and we obtain the second statement of the proposition.
\end{proof}

From this one can in principle compute the intrinsic torsion of $\Q'$ from that of $\Q$.
One only needs to know the reference connections $A_0$ and $A_0'$.
However, in the situation of the second part of this proposition, the change of intrinsic torsion is completely determined by $\zeta_h$ already.

\section{G-structures and instantons}
\label{sect:G_structures_and_instantons}

In the previous section we considered general principal $G$-subbundles of a principal fibre bundle and investigated normal deformations thereof.
Let us now specialise to the case where the ambient principal bundle is the frame bundle of $M$, i.e. $\P = F(M)$ and $H = GL(D,\mathbb{R})$.
Principal $G$-subbundles $\Q$ of $F(M)$ are then called \emph{$G$-structures} on $M$.

In particular, in the remainder of this paper we concentrate on metric $G$-structures, i.e. $G$-structures where $G \subset SO(D) \subset GL(D,\mathbb{R})$.
These $G$-structures are compatible with a Riemannian metric $g$ in the sense that $\Q \subset SO(M,g)$ is a principal subbundle of the bundle of orthonormal frames as defined by $g$.

Any Riemannian metric $g$ on $M$ induces an isomorphism
\begin{equation}
 I_g : Ad(SO(M,g)) \rightarrow \Lambda^2 T^*M
\end{equation}
from the adjoint bundle of $SO(M,g)$ to $\Lambda^2T^*M$ (cf.~\cite{HarlandNoelle}).

\begin{Def}
 The image of the restriction of $I_g$ to the adjoint bundle $Ad(\Q)$ of $\Q$ defines the \emph{instanton bundle} $W(\Q)$ associated to $\Q$.
 That is,
 \begin{equation}
  W(\Q) \coloneq I_g\big( Ad(\Q) \big).
 \end{equation}
 A connection $A \in \C(\B)$ on a principal bundle $(\B,\pi,M,B)$ over $M$ whose field strength satisfies
 \begin{equation}
  F_A \in \Gamma \big( W(\Q) \otimes Ad(\B) \big)
 \end{equation}
 is called an \emph{instanton for $\Q$}.
\end{Def}

In contrast to the case of generic connections as treated in the previous section, there seems to be no generic identification of instantons for $\Q$ and $\Q' = \R_h\, \Q$.
In particular, the isomorphism constructed in Theorem~\ref{st:general_isomp_of_connections} does, in general, not map instantons of $\Q$ to instantons of $\Q'$.
Thus, we choose to investigate a more constrained situation.
Both $G$-structures will certainly define the same instantons if their instanton bundles coincide, i.e.~if $W(\Q) = W(\Q')$.
In order to determine when this is the case, consider a local section $e \in \Gamma(U,\Q)$ of $\Q$.
Note that, since $\P = F(M)$, $e = (e_1,\ldots,e_D)$ is a local (orthonormal) frame.
In the above notation the local bases are defined via
\begin{equation}
 e_i = [e,v_i],
\end{equation}
where $(v_1,\ldots,v_D)$ is the standard orthonormal basis of $\mathbb{R}^D$.\\
A $2$-form $\omega$ is a section of $W(\Q)$ if and only if its components $\omega_{ij}$ with respect to local sections $e$ of $\Q$ form a matrix in $\g \subset \so(D)$.
Now, consider the local section $e' = R_h \circ e \in \Gamma(U,\Q')$.
As local frames, $e'_i = \rho(h)\indices{^j_i}\, e_j$, where $\rho$ is the standard representation of $GL(D,\mathbb{R})$ on $\mathbb{R}^D$.
Therefore, a section $\omega \in \Gamma(W(\Q))$ is a section of $W(\Q')$ if and only if the components of $\omega$ with respect to the local frame $e'$ form a matrix in $\g$.
If $\beta$ and $\beta'$ are the local coframes dual to $e$ and $e'$, respectively, we have
\begin{equation}
 \omega = \frac{1}{2}\, \omega'_{ij}\, \b'^i \wedge \b'^j = \frac{1}{2}\, \rho(h)\indices{^i_a}\, \omega'_{ij}\, \rho(h)\indices{^j_b}\, \b^a \wedge \b^b.
\end{equation}
From this we deduce

\begin{Prop}
\label{st:W(Q)=W(Q')}
 Let $\Q$ and $\Q'$ be two metric $G$-structures satisfying the normal deformation property with respect to $h \in C^\infty(M,N_H(G))$.
 Their instanton bundles coincide if and only if the map
 \begin{equation}
  \Phi_{h(x)}: \so(D) \rightarrow \so(D), \quad \omega_{ij} \mapsto \rho(h(x))\indices{^a_i}\, \omega_{ab}\, \rho(h(x))\indices{^b_j}
 \end{equation}
 preserves the subspace $\g \subset \so(D)$ for all $x \in M$.
 In particular, the instanton moduli spaces of $\Q$ and $\Q'$ coincide.
\end{Prop}

\noindent
For example, this holds true for $h$ being proportional to the identity matrix, i.e.~for conformal rescalings.
Furthermore, if $h$ is $SO(D)$-valued, the above action coincides with the adjoint action of $SO(D)$ on $\so(D)$.
Due to the normal deformation property, $h$ additionally takes values in the normaliser of $\g$, and, thus, this adjoint action preserves $\g$ in $\so(D)$.
Hence, if we have $h \in C^\infty(M,N_{SO(D)}(G))$ the above proposition always applies.\medskip

Let us finish this section by a short remark on the intrinsic torsion of metric $G$-structures.
Every metric $G$-structure $\Q$ is contained in a unique ambient orthonormal frame bundle. That is, we have inclusions
\begin{equation}
 \Q \subset SO(M,g) \subset F(M)
\end{equation}
for some metric $g$ on $M$.
The intrinsic torsion of $\Q$ is then defined as in \eqref{eq:inttor}, where we now choose $A_0$ to be the Levi-Civita connection of the metric it defines.
As pointed out in Section~\ref{sect:ndefs}, if $h$ is not $SO(D)$-valued, $\Q$ and $\Q'$ will define different metrics.
Thus, in this case, $\Q'$ will be contained in another $SO(M,g')$ with a different Riemannian metric $g'$, whence a different Levi-Civita connection is to be used in the computation of the intrinsic torsion of $\Q'$.

If, however, $h$ takes values in $SO(D)$ exclusively, we have $A_0 = A_0'$.
For normal deformations of torsion-free metric $G$-structures by $SO(D)$-valued functions $h$, the intrinsic torsion of the deformed structure is, therefore, given by $\zeta_h$ according to the second part of Proposition~\ref{st:inttor}.

\section{Examples}

\subsection{Conformal rescalings}

These deformations are induced by $h = \phi\, \mathbbm{1}_D$, where $\phi \in C^\infty(M,\mathbb{R}_+)$ is a positive function on $M$.
For Theorem~\ref{st:general_isomp_of_connections} to apply we have to ensure that the splitting $\gl(D,\mathbb{R}) = \g \oplus \widetilde{\m}$ is invariant under $G$.
Note that the splitting $\gl(D,\mathbb{R}) = \so(D) \oplus \mathrm{sym}$ of $\gl(D,\mathbb{R})$ into antisymmetric and symmetric matrices is invariant under the adjoint action of $SO(D)$ and that $\widetilde{\m} = \m \oplus\mathrm{sym}$, for a splitting $\so(D) = \g \oplus \m$.
Therefore, if $G \subset SO(D)$, the splitting $\gl(D,\mathbb{R}) = \g \oplus \widetilde{\m}$ is invariant under $G$ if and only if there is a $G$-invariant splitting $\mathfrak{so}(D) = \g \oplus \m$.

If $G \subset SO(D)$, we have $pr_\g \circ h^*\mu_H = 0$ in addition.
This is because the matrix part of $h^*\mu_H = \phi^{-1}\, \d \phi \otimes \mathbbm{1}_D$ is symmetric, whereas $\g \subset \so(D)$ contains antisymmetric matrices only.
Thus, the two isomorphisms between $\C(\Q)$ and $\C(\Q')$ constructed in Theorem~\ref{st:general_isomp_of_connections} and Lemma~\ref{st:central_isomorphism_of_connections} coincide in this situation.

Explicitly, let $\Gamma$ be a connection on $\Q$, i.e.~the covariant derivative it induces preserves the $G$-structure.
Assume that it has the local representation
\begin{equation}
 e^*\Gamma = (\Gamma\indices{^i_j}) \in \Omega^1(U,\g)
\end{equation}
with respect to a local section $e \in \Gamma(U,\Q)$.
We are then given a local section $e' = (\phi\, e_1,\ldots,\phi\, e_D)$ of $\Q'$.
Theorem~\ref{st:general_isomp_of_connections} and Lemma~\ref{st:central_isomorphism_of_connections} imply that there exists a connection $\Gamma'$ on $\Q'$, i.e.~preserving $\Q'$, having local representations
\begin{equation}
 e'^*\Gamma' = ({\Gamma'}\indices{^i_j}) = (\Gamma\indices{^i_j}) \in \Omega^1(U,\g)
\end{equation}
with respect to $e'$.
Note that, although the local representations of $\Gamma$ and $\Gamma'$ coincide, these are representations with respect to different local sections and, therefore, represent different connections on $F(M)$ in general.
Alternatively, expressing both connections with respect to $e'$,
\begin{equation}
\big( e'^*{\Gamma'} \big)\indices{^j_k} = \big( e'^*{\Gamma} + (h^{-1})^*\mu_{GL} \big)\indices{^j_k}
 = \big( e'^*{\Gamma} \big)\indices{^j_k} - \underbrace{\d \log(\phi)\, \delta\indices{^j_k}}_{= \zeta_h},
\end{equation}
as local Lie-algebra-valued $1$-forms.
If $\nabla$ and $\nabla'$ are the covariant derivatives on $TM$ induced by $\Gamma$ and $\Gamma'$, respectively, one can check that
\begin{align}
 (\nabla'_{e'_i}\, g')(e'_j,e'_k) ={}& - g'(\nabla'_{e'_i}\, e'_j, e'_k) - g'(e'_j, \nabla'_{e'_i}\, e'_k) \notag\\
  ={}& - \phi ( \Gamma\indices{_i^l_j}\, \delta_{lk} + \Gamma\indices{_i^l_k}\, \delta_{lj} ) \notag\\
  ={}& \phi\, (\nabla_{e_i}\, g)(e_j,e_k)\\
  ={}& 0. \notag
\end{align}
In the second last identity we made use of the fact that the local connection forms of $\nabla$ and $\nabla'$ coincide in the respective bases.
In a similar manner all defining sections of $\Q$ are preserved by $\Gamma$ if and only if the defining sections of $\Q'$ are preserved by $\Gamma'$.
Moreover, Proposition~\ref{st:W(Q)=W(Q')} applies to this setting whence $\Gamma$ is an instanton for $\Q$ if and only if $\Gamma'$ is an instanton for $\Q'$.

\subsection{Constant deformations}

Another very simple example is the case where $h$ is a constant map.
Here, $h$ is not necessarily central for $G$, and we have $h^*\mu_H = 0$ whence the the inhomogeneous term in $f_{\Q,h}$ vanishes identically.
The condition of Proposition~\ref{st:W(Q)=W(Q')} is a non-trivial restriction on $h$.
However, if it is satisfied, $f_{\Q,h}$ maps instantons for $\Q$ to instantons for $\Q'$.
The local representations of $\Gamma$ and $\Gamma'$ are related as
\begin{equation}
\big( e'^*{\Gamma'} \big)\indices{^j_k} = \rho(h^{-1})\indices{^j_m}\, \big( e^*{\Gamma} \big)\indices{^m_n}\, \rho(h)\indices{^n_k}.
\end{equation}
Note that $\Gamma'$ is compatible with the $G$-structure $\Q'$.
Such a deformation occurred in~\cite{Bunk:2014coa}, where it led to constructions of half-flat $SU(3)$-structures on cylinders over $5$-dimensional Sasaki-Einstein $SU(2)$-structure manifolds.

\subsection{Central deformation of Sasaki-Einstein SU(2)}

In~\cite{Bunk:2014coa}, a point-dependent deformation of the lift of a Sasaki-Einstein $SU(2)$-structure to the sine-cone over $M^5$ has been introduced.
The generating $h$ was shown to be central for $SU(2)$ as well as orthonormal.
Thus, Proposition~\ref{st:W(Q)=W(Q')} applies, and the transformations map the instantons of $\Q$ to those of $\Q'$ and vice-versa.
The deformation proved to reproduce the nearly Kähler $SU(3)$-structure on the sine-cone which was originally constructed in~\cite{Fernandezea} using geometric flow equations.

\section{Conclusions and outlook}

In this paper, we introduced and investigated an action of $C^\infty(M,N_H(G))$ on the space $G\P$ of principal $G$-subbundles of an ambient principal $H$-bundle $\P$.
Under slightly stronger restrictions on $H$ and $G$, we were able to construct bijections between the connections on principal $G$-subbundles of $\P$ related by this action, and to deduce a formula for the intrinsic torsion of the resulting subbundles.
However, it certainly would be interesting to specify the stabiliser subgroup of a certain principal $G$-subbundle of $\P$ in $C^\infty(M,N_H(G))$.

We then specialised our considerations to $G$-structures and instantons, and found a condition on the deformations such that $G$-structures they relate define the same instantons.
Generalising this observation, one might ask for the full classification of $G$-structures that define the same instanton bundle.

Since the families of $G$-structures, or principal subbundles, constructed this way are controlled to some extent by the features of the original $G$-structure, it is tempting to ask how other important features of these structures change under normal deformations.
Proposition~\ref{st:inttor} is a partial result regarding the intrinsic torsion.
For a fully general statement on metric $G$-structures, however, one would have to compute how the Levi-Civita connection of the ambient orthonormal frame bundle of a metric $G$-structure changes under the action of $h$.
Furthermore, for applications it would be necessary to translate the statement made here to intrinsic torsion classes (in the sense of~\cite{Chiossi:2002tw,GrayHervella}).
In particular, if one could find a general formula for the changes in the torsion classes induced by a normal deformation, one could use it to find normal deformations which produce particularly interesting geometries explicitly.
Another, perhaps less accessible, goal would of course be to investigate the change of the instanton moduli spaces under more general deformations of $G$-structures than considered in Proposition~\ref{st:W(Q)=W(Q')}.

The method of flow equations, as introduced in~\cite{Hitchin:2001}, is another way of constructing new $G$-structures from given ones.
In one dimension higher it yields $G'$-structures of well-known torsion type; the torsion of the resulting structure can be controlled by the design of the flow equations.
However, it is in general unclear how the torsion of the $G$-structure evolves under the flow, in contrast to the results obtained here for normal deformations.
This might be evidence of the richer nature of flow equations, which presumably encompass a larger variety of curves in the $G$-structure moduli space of $M$.
Nevertheless, some recent progress in the investigation of the evolution of torsion classes under flow equations has been made in~\cite{delaOssa:2014l}.

The difficulty with flow equations is the construction of solutions.
Here there might be a promising application of normal deformations.
A solution to a flow equation is a family $\Q_t$ of $G$-structures.
In order to find solutions to a flow equation, one might consider ansätze of the form $\Q_t = \R_{h_t} \Q$ for a family $h_t \in C^\infty(M,N_{GL(D)}(G))$.
That is, families of normal deformations of a $G$-structure on $M$ could yield interesting $G'$-structures on warped products $M \times I$, where $I$ is an interval.
Here it would again be useful to have an expression for the change of torsion classes for particular $G$-structures under normal deformations, as one could then more easily try to apply normal deformations to flow equations formulated in terms of torsion classes as for instance in~\cite{delaOssa:2014l}.
Examples where normal deformations have been employed successfully in the construction of $G'$-structures on $M \times I$ can be found in~\cite{Bunk:2014coa}.

In conclusion, normal deformations of $G$-structures provide a tool for constructing new geometries which is easier to handle than flow equations at the cost of being less general.
Nevertheless, they can be applied in the construction of new explicit $G$-structures which may well be of interest in geometry as well as string theory.

\section*{Acknowledgements}

The author would like to thank Olaf Lechtenfeld, Alexander Popov, Marcus Sperling and Lutz Habermann for helpful discussions and support.

\bibliographystyle{amsplainurl}
\bibliography{NDefsBib}

\end{document}